\documentclass{amsproc}
\usepackage{amsmath}
\usepackage{amssymb}
\usepackage{amsfonts}
\usepackage{hyperref}

\newtheorem{theorem}{Theorem}[section]
\newtheorem{lemma}[theorem]{Lemma}
\theoremstyle{definition}
\newtheorem{definition}[theorem]{Definition}
\theoremstyle{proposition}
\newtheorem{proposition}[theorem]{Proposition}

\theoremstyle{remark}

\numberwithin{equation}{section}
\renewcommand{\subjclassname}{  \textup{2010} Mathematics Subject Classification} 
\input{tcilatex}

\begin{document}
\title{$G_{2}$-structure deformations and warped products}
\author[S. Grigorian]{Sergey Grigorian}
\address{Simons Center for Geometry and Physics\\
Stony Brook University\\
Stony Brook, NY 11794\\
USA}
\date{September 26, 2011}

\begin{abstract}
We overview the properties of non-infinitesimal deformations of $G_{2}$%
-structures on seven-manifolds, and in particular, focus on deformations
that lie in the seven-dimensional representation of $G_{2}$ and are thus
defined by a vector. We then consider deformations from $G_{2}$-structures
with the torsion class having one-dimensional and seven-dimensional
components (so-called conformally nearly parallel $G_{2}$-manifolds) to $%
G_{2}$-structures with just a one-dimensional torsion component (nearly
parallel $G_{2}$-manifolds). We find that deformations between such
structures exist if and only if the metric is a particular warped product
metric.
\end{abstract}

\maketitle



%    author one information

%\curraddr{}
%\email{}
%\thanks{}

%    author two information
%\author{}
%\address{}
%\curraddr{}
%\email{}
%\thanks{}

%    The 2010 edition of the Mathematics Subject Classification is
%    now available.  If you are citing a classification from the
%    new scheme, use the following input coding instead.
%\subjclass[2010]{Primary }

\section{Introduction}

\setcounter{equation}{0}One of the most general geometric structures that
can be constructed on a $7$-dimensional manifold is a $G_{2}$-structure. A $%
G_{2}$-structure can be considered as a generalization of the vector cross
product on $\mathbb{R}^{7}$ \cite{Gray-VCP}. It is well-known that a $7$%
-manifold admits a $G_{2}$-structure if and only if it is orientable and
admits a spin structure, or equivalently, if the first two Stiefel-Whitney
classes vanish \cite{FernandezGray, FriedrichNPG2}. A very important special
case of a $G_{2}$-structure is a torsion-free $G_{2}$-structure. This
implies that the holonomy group lies in $G_{2}$. In Section \ref{secg2struct}
we give a more precise definition and an overview of the properties of $G_{2}
$-structures. The concept of $G_{2}$-structures also has important
applications in physics - as shown in \cite{Kaste:2003zd}, the most general
backgrounds for $M$-theory compactifications with fluxes are indeed $7$%
-manifolds with $G_{2}$-structures with some particular torsion.

Given a $7$-manifold with a $G_{2}$-structure defined by the $3$-form $%
\varphi $ with torsion $T$, a natural question to ask is whether we can
modify this $3$-form to get a new $G_{2}$-structure with torsion that lies a
strictly lower torsion class. If $\varphi $ is deformed by a $3$-form lying
in the $7$-dimensional component of $\Lambda ^{3}$, it is easy to see that
such a deformation will always yield a new $G_{2}$-structure and in my paper 
\cite{GrigorianG2Torsion1}, I have explicitly calculated the new torsion in
terms of the old one, and the derived the equation that $v$ must satisfy to
take a torsion $T$ to torsion $\tilde{T}$. It was moreover shown that on
closed, compact manifolds there are no such deformations from strict torsion
classes $W_{1}$, $W_{7}$, $W_{1}\oplus W_{7}$ to the vanishing torsion class 
$W_{0}$, and vice versa.

In this paper we use the general results from \cite{GrigorianG2Torsion1} and
apply them to the situation when we want a transition from the torsion class 
$W_{1}\oplus W_{7}$ to $W_{1}$. The torsion class $W_{1}\oplus W_{7}$ is
known as \emph{conformally nearly parallel }because a conformal
transformation takes $W_{1}\oplus W_{7}$ to $W_{1}$. Here, however, we show
that there exists a deformation of $\varphi $ in $\Lambda _{7}^{3}$ that
takes a torsion in $W_{1}\oplus W_{7}$ to $W_{1}$ if and only if the metric
is a particular warped product. However $G_{2}$-structures with such torsion
have been constructed by Cleyton and Ivanov in \cite{CleytonIvanovConf} as a
warped product of an interval over a nearly K\"{a}hler manifold, so these
examples fit as solutions.

\section{$G_{2}$-structures}

\setcounter{equation}{0}\label{secg2struct}The 14-dimensional group $G_{2}$
is the smallest of the five exceptional Lie groups and is closely related to
the octonions. In particular, $G_{2}$ can be defined as the automorphism
group of the octonion algebra. Taking the imaginary part of octonion
multiplication of the imaginary octonions defines a vector cross product on $%
V=\mathbb{R}^{7}$ and the group that preserves the vector cross product is
precisely $G_{2}$. A more detailed account of the relationship between
octonions and $G_{2}$ can be found in \cite{BaezOcto, GrigorianG2Review}.The
structure constants of the vector cross product define a $3$-form on $%
\mathbb{R}^{7}$, hence $G_{2}$ can alternatively be defined as the subgroup
of $GL\left( 7,\mathbb{R}\right) $ that preserves a particular $3$-form \cite%
{Joycebook}. In general, given a $n$-dimensional manifold $M$, a $G$%
-structure on $M$ for some Lie subgroup $G$ of $GL\left( n,\mathbb{R}\right) 
$ is a reduction of the frame bundle $F$ over $M$ to a principal subbundle $P
$ with fibre $G$. A $G_{2}$-structure is then a reduction of the frame
bundle on a $7$-dimensional manifold $M$ to a $G_{2}$ principal subbundle.
It turns out that there is a $1$-$1$ correspondence between $G_{2}$%
-structures on a $7$-manifold and smooth $3$-forms $\varphi $ for which the $%
7$-form-valued bilinear form $B_{\varphi }$ as defined by (\ref{Bphi}) is
positive definite (for more details, see  \cite{Bryant-1987} and the arXiv
version of \cite{Hitchin:2000jd}).  
\begin{equation}
B_{\varphi }\left( u,v\right) =\frac{1}{6}\left( u\lrcorner \varphi \right)
\wedge \left( v\lrcorner \varphi \right) \wedge \varphi   \label{Bphi}
\end{equation}%
Here the symbol $\lrcorner $ denotes contraction of a vector with the
differential form: 
\begin{equation*}
\left( u\lrcorner \varphi \right) _{mn}=u^{a}\varphi _{amn}.
\end{equation*}%
Note that we will also use this symbol for contractions of differential
forms using the metric.

A smooth $3$-form $\varphi $ is said to be \emph{positive }if $B_{\varphi }$
is the tensor product of a positive-definite bilinear form and a
nowhere-vanishing $7$-form. In this case, it defines a unique metric $%
g_{\varphi }$ and volume form $\mathrm{vol}$ such that for vectors $u$ and $%
v $, the following holds 
\begin{equation}
g_{\varphi }\left( u,v\right) \mathrm{vol}=\frac{1}{6}\left( u\lrcorner
\varphi \right) \wedge \left( v\lrcorner \varphi \right) \wedge \varphi
\label{gphi}
\end{equation}

In components we can rewrite this as 
\begin{equation}
\left( g_{\varphi }\right) _{ab}=\left( \det s\right) ^{-\frac{1}{9}}s_{ab}\ 
\text{where \ }s_{ab}=\frac{1}{144}\varphi _{amn}\varphi _{bpq}\varphi _{rst}%
\hat{\varepsilon}^{mnpqrst}.  \label{metricdefdirect}
\end{equation}%
Here $\hat{\varepsilon}^{mnpqrst}$ is the alternating symbol with $\hat{%
\varepsilon}^{12...7}=+1$. Following Joyce (\cite{Joycebook}), we will adopt
the following definition

\begin{definition}
Let $M$ be an oriented $7$-manifold. The pair $\left( \varphi ,g\right) $
for a positive $3$-form $\varphi $ and corresponding metric $g$ defined by (%
\ref{gphi}) will be referred to as a $G_{2}$-structure.
\end{definition}

Since a $G_{2}$-structure defines a metric and an orientation, it also
defines a Hodge star. Thus we can construct another $G_{2}$-invariant object
- the $4$-form $\ast \varphi $. Since the Hodge star is defined by the
metric, which in turn is defined by $\varphi $, the $4$-form $\ast \varphi $
depends non-linearly on $\varphi $. For convenience we will usually denote $%
\ast \varphi $ by $\psi $.

For a general $G$-structure, the spaces of $p$-forms decompose according to
irreducible representations of $G$. Given a $G_{2}$-structure, $2$-forms
split as $\Lambda ^{2}=\Lambda _{7}^{2}\oplus \Lambda _{14}^{2}$, where $%
\Lambda _{7}^{2}=\left\{ \alpha \lrcorner \varphi \text{: for a vector field 
}\alpha \right\} $ and 
\begin{equation*}
\Lambda _{14}^{2}=\left\{ \omega \in \Lambda ^{2}\text{: }\left( \omega
_{ab}\right) \in \mathfrak{g}_{2}\right\} =\left\{ \omega \in \Lambda ^{2}%
\text{: }\omega \lrcorner \varphi =0\right\} .
\end{equation*}%
The $3$-forms split as $\Lambda ^{3}=\Lambda _{1}^{3}\oplus \Lambda
_{7}^{3}\oplus \Lambda _{27}^{3}$, where the one-dimensional component
consists of forms proportional to $\varphi $, forms in the $7$-dimensional
component are defined by a vector field $\Lambda _{7}^{3}=\left\{ \alpha
\lrcorner \psi \text{: for a vector field }\alpha \right\} $, and forms in
the $27$-dimensional component are defined by traceless, symmetric matrices: 
\begin{equation}
\Lambda _{27}^{3}=\left\{ \chi \in \Lambda ^{3}:\chi
_{abc}=h_{[a}^{d}\varphi _{bc]d}\text{ for }h_{ab}~\text{traceless, symmetric%
}\right\} .  \label{lam327}
\end{equation}%
By Hodge duality, similar decompositions exist for $\Lambda ^{4}$ and $%
\Lambda ^{5}$. A detailed description of these representations is given in 
\cite{Bryant-1987,bryant-2003}. Also, formulae for projections of
differential forms onto the various components are derived in detail in \cite%
{GrigorianG2Torsion1, GrigorianYau1, karigiannis-2007}.

The \emph{intrinsic torsion }of a $G_{2}$-structure is defined by $\nabla
\varphi $, where $\nabla $ is the Levi-Civita connection for the metric $g$
that is defined by $\varphi $. Following \cite{karigiannis-2007}, it is easy
to see 
\begin{equation}
\nabla \varphi \in \Lambda _{7}^{1}\otimes \Lambda _{7}^{3}\cong W.
\label{torsphiW}
\end{equation}%
Here we define $W$ as the space $\Lambda _{7}^{1}\otimes \Lambda _{7}^{3}$.
Given (\ref{torsphiW}), we can write 
\begin{equation}
\nabla _{a}\varphi _{bcd}=T_{a}^{\ \ e}\psi _{ebcd}  \label{fulltorsion}
\end{equation}%
where $T_{ab}$ is the \emph{full torsion tensor}. From this we can also
write 
\begin{equation}
T_{a}^{\ m}=\frac{1}{24}\left( \nabla _{a}\varphi _{bcd}\right) \psi ^{mbcd}.
\label{tamphipsi}
\end{equation}%
This $2$-tensor fully defines $\nabla \varphi $ since pointwise, it has 49
components and the space $W$ is also 49-dimensional (pointwise). In general
we can split $T_{ab}$ according to representations of $G_{2}$ into \emph{%
torsion components}: 
\begin{equation}
T=\tau _{1}g+\tau _{7}\lrcorner \varphi +\tau _{14}+\tau _{27}
\label{torsioncomps}
\end{equation}%
where $\tau _{1}$ is a function, and gives the $\mathbf{1}$ component of $T$%
. We also have $\tau _{7}$, which is a $1$-form and hence gives the $\mathbf{%
7}$ component, and, $\tau _{14}\in \Lambda _{14}^{2}$ gives the $\mathbf{14}$
component and $\tau _{27}$ is traceless symmetric, giving the $\mathbf{27}$
component. Hence we can split $W$ as 
\begin{equation}
W=W_{1}\oplus W_{7}\oplus W_{14}\oplus W_{27}.  \label{Wsplit}
\end{equation}%
As it was originally shown by Fern\'{a}ndez and Gray \cite{FernandezGray},
there are in fact a total of 16 torsion classes of $G_{2}$-structures that
arise as the subsets of $W$ to which $\nabla \varphi $ belongs. Moreover, as
shown in \cite{karigiannis-2007}, the torsion components $\tau _{i}$ relate
directly to the expression for $d\varphi $ and $d\psi $. In fact, in our
notation, 
%TCIMACRO{\TeXButton{TeX field}{\begin{subequations}}}%
%BeginExpansion
\begin{subequations}%
%EndExpansion
\label{dptors} 
\begin{eqnarray}
d\varphi &=&4\tau _{1}\psi -3\tau _{7}\wedge \varphi -\ast \tau _{27}
\label{dphitors} \\
d\psi &=&-4\tau _{7}\wedge \psi -2\ast \tau _{14}.  \label{dpsitors}
\end{eqnarray}%
%TCIMACRO{\TeXButton{TeX field}{\end{subequations}}}%
%BeginExpansion
\end{subequations}%
%EndExpansion
Note that in the literature (\cite{bryant-2003,CleytonIvanovConf}, for
example) a slightly different convention for torsion components is sometimes
used. Our $\tau _{1}$ then corresponds to $\frac{1}{4}\tau _{0}$, $\tau _{7}$
corresponds to $-\tau _{1}$ in their notation, $\tau _{27}$ corresponds to $%
-\tau _{3}$ and $\tau _{14}$ corresponds to $-\frac{1}{2}\tau _{2}$.
Similarly, our torsion classes $W_{1}\oplus W_{7}\oplus W_{14}\oplus W_{27}$
correspond to $W_{0}\oplus W_{1}\oplus W_{2}\oplus W_{3}$.

An important special case is when the $G_{2}$-structure is said to be
torsion-free, that is, $T=0$. This is equivalent to $\nabla \varphi =0$ and
also equivalent, by Fern\'{a}ndez and Gray, to $d\varphi =d\psi =0$.
Moreover, a $G_{2}$-structure is torsion-free if and only if the holonomy of
the corresponding metric is contained in $G_{2}$ \cite{Joycebook}. The
holonomy group is then precisely equal to $G_{2}$ if and only if the
fundamental group $\pi _{1}$ is finite.

The torsion tensor $T_{ab}$ and hence the individual components $\tau _{1}$,$%
\tau _{7}$,$\tau _{14}$ and $\tau _{27}$ must also satisfy certain
differential conditions. For the exterior derivative $d$, $d^{2}=0$, so from
(\ref{dptors}), must have%
%TCIMACRO{\TeXButton{TeX field}{\begin{subequations}}}%
%BeginExpansion
\begin{subequations}%
%EndExpansion
\label{dtabcond} 
\begin{eqnarray}
d\left( 4\tau _{1}\psi -3\tau _{7}\wedge \varphi -\ast \tau _{27}\right) &=&0
\\
d\left( 4\tau _{7}\wedge \psi +2\ast \tau _{14}\right) &=&0
\end{eqnarray}%
%TCIMACRO{\TeXButton{TeX field}{\end{subequations}}}%
%BeginExpansion
\end{subequations}%
%EndExpansion
These conditions are explored in more detail in \cite{GrigorianG2Torsion1}.
However for some simple torsion classes, the conditions on torsion
components simplify. In particular, if the torsion is in the class $W_{1}$,
so that only the $\tau _{1}$ component is non-vanishing, we just get the
condition ${\small d\tau }_{1}{\small =0}$. Similarly, for the $W_{7}$ class
the condition is ${\small d\tau }_{7}{\small =0}$. For the class ${\small W}%
_{1}{\small \oplus W}_{7}$, both $\tau _{1}$ and $\tau _{7}$ are
non-vanishing, and the condition is $d\tau _{1}=\tau _{1}\tau _{7}$. So if $%
\tau _{1}$ is nowhere zero, we have 
\begin{equation}
\tau _{7}=d\left( \log \tau _{1}\right) .  \label{w1w7torsioncond}
\end{equation}%
However if $\tau _{1}$ does vanish somewhere, it was shown in \cite%
{CleytonIvanovConf} that it must in fact vanish identically, and so the
torsion class reduces to $W_{7}$.

\section{Deformations of $G_{2}$-structures}

\label{secdeform}\setcounter{equation}{0}Suppose we have a $G_{2}$-structure
on $M$ defined by the $3$-form $\varphi $, and we want to obtain a new $%
G_{2} $-structure $\tilde{\varphi}$ by adding another $3$-form $\chi $ 
\begin{equation}
\varphi \longrightarrow \tilde{\varphi}=\varphi +\chi  \label{phideform1}
\end{equation}%
There are a number of challenges associated with this. Firstly, for a
generic $3$-form $\chi $, the $3$-form $\tilde{\varphi}$ may not even define
a $G_{2}$-structure. In order for $\tilde{\varphi}$ to define a $G_{2}$%
-structure it has to be a positive $3$-form. In this case, as shown in (\cite%
{GrigorianYau1}), $\tilde{\varphi}$ defines a Riemannian metric $\tilde{g}$
given by 
\begin{equation}
\tilde{g}_{ab}=\left( \frac{\det g}{\det \tilde{g}}\right) ^{\frac{1}{2}}%
\tilde{s}_{ab}  \label{gtildedeform1}
\end{equation}%
for 
\begin{equation}
\tilde{s}_{ab}=g_{ab}+\frac{1}{2}\chi _{mn(a}\varphi _{b)}^{\ \ mn}+\frac{1}{%
8}\chi _{amn}\chi _{bpq}\psi ^{mnpq}+\frac{1}{24}\chi _{amn}\chi
_{bpq}\left( \ast \chi \right) ^{mnpq}  \label{sabgen}
\end{equation}%
In fact, $\tilde{\varphi}$ is positive if and only if $\tilde{s}$ is
positive-definite, so in general this gives some algebraic constraints on $%
\chi $. If we assume that $\tilde{\varphi}$ does in fact define a $G_{2}$%
-structure, the next question is the torsion class of the new $G_{2}$%
-structure. The metric $\tilde{g}$ defines a Levi-Civita connection $\tilde{%
\nabla}$, so the new torsion (with lowered indices) is 
\begin{equation}
\tilde{T}_{am}^{\ }=\frac{1}{24}\left( \tilde{\nabla}_{a}\tilde{\varphi}%
_{bcd}\right) \tilde{\psi}_{m}^{\ \ \tilde{b}\tilde{c}\tilde{d}}
\label{tabtilde}
\end{equation}%
Here $\tilde{\psi}=\tilde{\ast}\tilde{\varphi}$, the Hodge dual of $\tilde{%
\varphi}$ with the Hodge star $\tilde{\ast}$ being defined by the metric $%
\tilde{g}$. The tilded raised indices on $\tilde{\psi}$ denote indices
raised by $\tilde{g}$. In \cite{GrigorianG2Torsion1}, I derived an explicit
expression for $\tilde{T}$ in terms of the old torsion $T$ and the $3$-form $%
\chi $:

\begin{proposition}[\protect\cite{GrigorianG2Torsion1}]
Given a deformation of $\varphi $ as in (\ref{phideform1}), the full torsion 
$\tilde{T}$ of the new $G_{2}$-structure $\tilde{\varphi}$ is given by%
\begin{eqnarray}
\tilde{T}_{an} &=&\frac{1}{24}\left( \frac{\det g}{\det \tilde{g}}\right)
\left( \left( 24T_{a}^{\ m}+T_{a}^{\ e}\psi _{ebcd}\left( \ast \chi \right)
^{mbcd}+\right. \right.  \notag \\
&&\left. +\psi ^{mbcd}\nabla _{a}\chi _{bcd}+\nabla _{a}\chi _{bcd}\left(
\ast \chi \right) ^{mbcd}\right) \tilde{s}_{mn}  \label{Tanlow2} \\
&&-3\left( 4\varphi _{c}^{\ \ bd}+\varphi _{cpq}\ast \chi ^{pqbd}+\chi
_{cpq}\psi ^{pqbd}+\chi _{cpq}\left( \ast \chi \right) ^{pqbd}\right) \times
\notag \\
&&\left. \times \left( \delta _{n}^{c}\nabla _{b}s_{ad}-\frac{1}{9}\delta
_{a}^{c}\tilde{g}_{bn}\tilde{g}^{\tilde{p}\tilde{q}}\nabla _{d}\tilde{s}%
_{pq}\right) \right) .  \notag
\end{eqnarray}%
where $\tilde{s}_{ab}$ is given by (\ref{sabgen}).
\end{proposition}

Using (\ref{Tanlow2}), it is then possible to extract the components of $%
\tilde{T}$ in $\tilde{W}_{1}\oplus \tilde{W}_{7}\oplus \tilde{W}_{14}\oplus 
\tilde{W}_{27}$ and hence determine the new torsion class. Note that since
the $G_{2}$-structure is different, the decomposition $\tilde{W}_{1}\oplus 
\tilde{W}_{7}\oplus \tilde{W}_{14}\oplus \tilde{W}_{27}$ differs from $%
W_{1}\oplus W_{7}\oplus W_{14}\oplus W_{27}$. An interesting question is
whether, given a $G_{2}$-structure in a specific torsion class, we can find
a $3$-form $\chi $ such that the new $G_{2}$-structure is in a strictly
smaller torsion class. From (\ref{Tanlow2}), this obviously involves solving
a non-linear differential equation for $\chi $, subject to algebraic
constraints that (\ref{sabgen}) is positive-definite. One way to simplify
the problem is to restrict the choice of $\chi .$ Using the original $G_{2}$%
-structure $\varphi $ we can decompose the $3$-form $\chi $ according to
representations of $G_{2}$. So in general it has a $\Lambda _{1}^{3}$
component that is proportional to $\varphi $, a $\Lambda _{7}^{3}$ component
that is of the form $v\lrcorner \psi $ for some vector $v$ and a more
complicated $\Lambda _{27}^{3}$ component. For a generic $\chi $, many of
the difficulties come from the $\Lambda _{27}^{3}$ component. These of
course can be avoided if we only consider deformations by $3$-forms that
have components only in either $\Lambda _{1}^{3}$ or $\Lambda _{7}^{3}$.

A deformation by a $3$-form in $\Lambda _{1}^{3}$ is equivalent to a
conformal transformation. So let $\chi =\left( f^{3}-1\right) \varphi $, so
that $\tilde{\varphi}=f^{3}\varphi $. Clearly $\tilde{\varphi}$ still
defines a $G_{2}$-structure. Then from (\ref{sabgen}) and (\ref%
{gtildedeform1}) we get 
\begin{equation}
s_{ab}=f^{9}g_{ab}\ \text{and thus, }\tilde{g}_{ab}=f^{2}g_{ab}.
\label{sconf}
\end{equation}%
Substituting into (\ref{Tanlow2}), we find that 
\begin{equation}
\tilde{T}=fT-df\lrcorner \varphi  \label{conftranstors}
\end{equation}%
In particular, we see that such a transformation only affects the $W_{7}$
component of the torsion. Moreover, if $T$ has a $W_{7}$ component $\tau
_{7} $ that is an exact form, then we can always find a function $f$ so that
a conformal transformation will remove this torsion component.

As an example, suppose $\varphi $ has torsion in the strict class $%
W_{1}\oplus W_{7}$. Then from (\ref{w1w7torsioncond}), we know that $\tau
_{7}=d\left( \log \tau _{1}\right) $. Hence if we take $f=\frac{\tau _{1}}{C}
$ for any non-zero constant $C$, the new torsion will be 
\begin{equation}
\tilde{T}=\frac{\tau _{1}}{C}\tau _{1}g=C\left( \frac{\tau _{1}}{C}\right)
^{2}g=C\tilde{g}
\end{equation}%
and thus in the $\tilde{W}_{1}$ class. Therefore, the conformal
transformation $\tilde{\varphi}=\left( \frac{\tau _{1}}{C}\right)
^{3}\varphi $ reduces the class $W_{1}\oplus W_{7}$ to $W_{1}$. Conversely,
a conformal transformation of the $W_{1}$ class will result in $W_{1}\oplus
W_{7}$. Since $G_{2}$-structures in the $W_{1}$ class are sometimes called 
\emph{nearly parallel}, the $G_{2}$-structures in the strict $W_{1}\oplus
W_{7}$ class are referred to as \emph{conformally nearly parallel}. If $%
W_{1}=0$, then we just have the $W_{7}$ class. In this case, we know that $%
\tau _{7}$ is closed. So by the Poincar\'{e} Lemma, we can at least locally
find a function $h$ such that $dh=\tau _{7}\,.$ By taking a conformal
transformation with $f=e^{h}$, we can thus remove the torsion locally. Hence
the $W_{7}$ class is sometimes called \emph{locally conformally parallel. }

Now suppose we look at deformations where $\chi _{bcd}=v^{e}\psi _{bcde}^{\
\ \ \ \ \ }\in \Lambda _{7}^{3}$. It was shown by Karigiannis in \cite%
{karigiannis-2005-57}, that in this case, if we let $\left\vert v\right\vert
^{2}=M$, with respect to the old metric $g$, 
%TCIMACRO{\TeXButton{TeX field}{\begin{subequations}}}%
%BeginExpansion
\begin{subequations}%
%EndExpansion
\label{pi7quants} 
\begin{eqnarray}
s_{ab} &=&\left( 1+M\right) g_{ab}-v_{a}v_{b}  \label{pi7sab} \\
\left( \frac{\det \tilde{g}}{\det g}\right) ^{\frac{1}{2}} &=&\left(
1+M\right) ^{\frac{2}{3}}  \label{pi7gtilde} \\
\tilde{g}_{ab} &=&\left( 1+M\right) ^{-\frac{2}{3}}\left( \left( 1+M\right)
g_{ab}-v_{a}v_{b}\right)   \label{pi7gdown} \\
\tilde{g}^{\tilde{a}\tilde{b}} &=&\left( 1+M\right) ^{-\frac{1}{3}}\left(
g^{mu}+v^{m}v^{u}\right)   \label{pi7gup}
\end{eqnarray}

%TCIMACRO{\TeXButton{TeX field}{\end{subequations}} }%
%BeginExpansion
\end{subequations}
%EndExpansion
Note that the deformed metric defined above is always positive definite. To
see this, suppose $\xi ^{a}$ is some vector, then 
\begin{equation}
\tilde{g}_{ab}\xi ^{a}\xi ^{b}=\left( 1+\left\vert v\right\vert ^{2}\right)
^{-\frac{2}{3}}\left( \left\vert \xi \right\vert ^{2}+\left\vert
v\right\vert ^{2}\left\vert \xi \right\vert ^{2}-\left( v_{a}\xi ^{a}\right)
^{2}\right) \geq 0  \label{pi7deformposdef}
\end{equation}%
since $\left( v_{a}\xi ^{a}\right) ^{2}\leq \left\vert v\right\vert
^{2}\left\vert \xi \right\vert ^{2}$. Therefore, under such a deformation,
the $3$-form $\tilde{\varphi}$ is always a positive $3$-form, and thus
indeed defines a $G_{2}$-structure.

In \cite{GrigorianG2Torsion1}, the expression for the new torsion was
derived using (\ref{pi7quants}) and (\ref{Tanlow2}). This is a very long and
messy expression, which we will not reproduce here, but it gives the new
torsion in terms of the old torsion components and $\nabla v$, which was
also decomposed according to $G_{2}$-representations as 
\begin{equation}
\nabla v=v_{1}g+v_{7}\lrcorner \varphi +v_{14}+v_{27}  \label{delvsplit}
\end{equation}%
The expression for $\tilde{T}$ was then inverted to obtain equations for $%
v_{1}$, $v_{7}$, $v_{14}$ and $v_{27}$ in terms of the old and new torsion
components. By analyzing the equations for $\nabla v$ in the case when the
original torsion lies in the class $W_{1}\oplus W_{7}$, it was shown that
the new torsion vanishes if and only if the original $G_{2}$-structure was
also torsion free, and moreover $\nabla v=0$. Similarly it was shown that
there are no deformations of this type which preserve the strict $W_{1}$
torsion class. Here we will attempt something different - what if the
torsion is in the class $W_{1}\oplus W_{7}$ and we want to obtain the class $%
W_{1}$. From the expression (\ref{conftranstors}) we already know that this
is possible to do with a conformal transformation. However if it were
possible to go from $W_{1}\oplus W_{7}$ to $W_{1}$ using $\chi \in \Lambda
_{7}^{3}$, then a composition of the two types of deformation would actually
give a much more complicated and interesting deformation that preserves the
class $W_{1}$.

\section{Conformally nearly parallel $G_{2}$-structures}

\label{seclam7deform}Suppose now we have a $G_{2}$-structure $\left( \varphi
,g\right) $ with torsion lying in the strict class $W_{1}\oplus W_{7}$ -that
is, both $\tau _{1}$ and $\tau _{7}$ are non-zero. We then deform $\varphi $
to $\tilde{\varphi}$ given by%
\begin{equation}
\tilde{\varphi}=\varphi +v^{e}\psi _{bcde}^{\ \ \ \ \ \ }.  \label{g2deform}
\end{equation}%
As we know from (\ref{pi7deformposdef}), the metric defined by $\tilde{%
\varphi}$ is positive definite, so $\tilde{\varphi}$ does indeed define a $%
G_{2}$-structure. From \cite{GrigorianG2Torsion1}, we can also write down
the torsion components of $\tilde{\varphi}$. As before, $M=\left\vert
v\right\vert ^{2}$ with respect to the old metric $g$, and $\nabla v$ is
decomposed into components as in (\ref{delvsplit}). Here we show the
expression for $\tilde{\tau}_{1}$ and $\tilde{\tau}_{7}$, the $1$- and $7$%
-dimensional components of the new torsion $\tilde{T}_{ab}$: 
\begin{eqnarray*}
\ \ \text{\ \ \ \ \ \ \ }\tilde{\tau}_{1} &=&\frac{\left( \left( 1+\frac{1}{7%
}M\right) \tau _{1}-v_{1}-\frac{6}{7}\left( \tau _{7}\right) ^{a}v_{a}+\frac{%
3}{7}\left( v_{7}\right) ^{a}v_{a}\right) }{\left( 1+M\right) ^{\frac{2}{3}}}
\\
\left( \tilde{\tau}_{7}\right) _{c} &=&\left( \tau _{7}\right) _{c}-\frac{1}{%
6}\varphi _{c}^{\ \ ab}\left( \tau _{7}\right) _{a}v_{b}+\frac{v_{c}\left(
6\tau _{1}-6\left( \tau _{7}\right) _{a}v^{a}-8v_{1}+3\left( v_{7}\right)
_{a}v^{a}\right) }{6\left( 1+M\right) } \\
&&-\frac{\left( 3\left( M+2\right) \left( v_{7}\right) _{c}+v^{a}\left(
v_{27}\right) _{ac}+\varphi _{ca}^{\ \ \ b}v^{a}\left( v_{27}\right)
_{bd}v^{d}+3\varphi _{cab}v^{a}\left( v_{7}\right) ^{b}\right) }{6\left(
1+M\right) }
\end{eqnarray*}%
The expressions for $\tilde{\tau}_{14}$ and $\tilde{\tau}_{27}$ can
similarly be written down in terms of $\tau _{1}$, $\tau _{7}$ and the
components of $\nabla v$, using the general results in \cite%
{GrigorianG2Torsion1}, however they are rather long and not very
enlightening. Also, as shown in \cite{GrigorianG2Torsion1}, the linear
equations for the torsion components can be solved for the components of $%
\nabla v$ in terms of the old torsion and the new torsion. Hence if we
require the new torsion to be in a specific torsion class, this would give
us a differential equation that $v$ has to satisfy. We will try to find
conditions that will have the $G_{2}$-structure in $W_{1}\oplus W_{7}$ class
to a $G_{2}$-structure that only has a $1$-dimensional torsion component.
Using the general expression for $\nabla v$ in \cite{GrigorianG2Torsion1},
and setting $\tau _{14}=\tau _{27}=0$ and $\tilde{\tau}_{7}=\tilde{\tau}%
_{14}=\tilde{\tau}_{27}=0$, we thus have:

\begin{proposition}
Suppose $\left( \varphi ,g\right) $ is $G_{2}$-structure with the only
non-vanishing torsion component $\tau _{1}$ and $\tau _{7}$. The new $G_{2}$%
-structure $\tilde{\varphi}$ obtained via the deformation (\ref{g2deform})
then has torsion in class $\tilde{W}_{1}$ with the non-vanishing component $%
\tilde{\tau}_{1}$ if and only if $v$ satisfies: 
\begin{eqnarray}
\ \ \ \ \ \ \ \ \nabla _{a}v_{b} &=&\left( \tau _{1}-\left( 1+M\right) ^{%
\frac{2}{3}}\tilde{\tau}_{1}-\left( \tau _{7}\right) _{c}v^{c}\right) g_{ab}
\label{fulldelvt17tt1} \\
&&+4\left( 1+M\right) ^{-\frac{1}{3}}\tilde{\tau}_{1}v_{a}v_{b}+\frac{1}{%
\left( M+9\right) }\left( -3\left( M-3\right) \left( \tau _{7}\right)
_{c}\varphi _{\ \ ab}^{c}\right.  \notag \\
&&-\left( M+33\right) v_{a}\left( \tau _{7}\right) _{b}+3\left( 1+M\right)
\left( \tau _{7}\right) _{a}v_{b}  \notag \\
&&-\frac{1}{3}v^{c}\varphi _{cab}\left( 9\tau _{1}-4\tilde{\tau}_{1}\left(
M+9\right) \left( 1+M\right) ^{-\frac{1}{3}}+\tau _{1}M-12\left( \tau
_{7}\right) _{d}v^{d}\right)  \notag \\
&&\left. +12v_{a}\varphi _{\ \ \ b}^{cd}\left( \tau _{7}\right)
_{c}v_{d}-12v_{b}\varphi _{\ \ \ a}^{cd}\left( \tau _{7}\right)
_{c}v_{d}+12\left( \tau _{7}\right) _{c}v_{d}\psi _{\ \ \ ab}^{cd}\right) 
\notag
\end{eqnarray}%
\qquad
\end{proposition}

Note that $\tilde{\tau}_{1}$ has to be a constant due to the conditions on
the torsion (\ref{dtabcond}). While it is too difficult to solve equation (%
\ref{fulldelvt17tt1}) directly, we can obtain conditions under which the
equation is at least consistent. In this equation $v$ has lowered indices,
so we can consider this as a $1$-form $v^{\flat }$. Then 
\begin{equation*}
\left( dv^{\flat }\right) _{ab}=2\nabla _{\lbrack a}v_{b]}
\end{equation*}%
However $d^{2}=0$, and thus the exterior derivative applied to the
anti-symmetrization of (\ref{fulldelvt17tt1}) must give zero. From these
considerations we get the consistency conditions in Proposition \ref%
{propddvcond} below. The extra equations which we get from the consistency
conditions are very important to simplify the equation (\ref{fulldelvt17tt1}%
).

\begin{proposition}
\label{propddvcond}The equation (\ref{fulldelvt17tt1}) is consistent with
the necessary condition $d^{2}v^{\flat }=0$ if and only if all of the
following conditions are satisfied:

\begin{enumerate}
\item For some smooth function $V$, 
\begin{equation}
v^{\flat }=V\tau _{7}  \label{vproptau7}
\end{equation}

\item The $7$-dimensional component $\tau _{7}$ of the original torsion
satisfies 
\begin{equation}
\nabla \tau _{7}=-\frac{1}{4}\frac{\left( V^{2}\left\vert \tau
_{7}\right\vert ^{2}-3\right) \left( V\tau _{1}+1\right) }{V^{2}}g+\frac{1}{6%
}\left( V^{2}\tau _{1}^{2}+6V\tau _{1}+3\right) \tau _{7}\otimes \tau _{7}
\label{deltau7eq}
\end{equation}

\item The $1$-dimensional component $\tilde{\tau}_{1}$ of the new torsion
satisfies 
\begin{equation}
\tilde{\tau}_{1}=\frac{1}{4V}\left( 1+V^{2}\left\vert \tau _{7}\right\vert
^{2}\right) ^{\frac{1}{3}}\left( V\tau _{1}-3\right)  \label{ttildecond}
\end{equation}
\end{enumerate}
\end{proposition}

\begin{proof}
As outlined above, we apply the condition $\nabla _{\lbrack a}\nabla
_{b}v_{c]}=0$ to (\ref{fulldelvt17tt1}). During the simplification process
we apply (\ref{fulldelvt17tt1}) again, and moreover use the conditions $%
d\tau _{7}=0$, $d\tau _{1}=\tau _{1}\tau _{7}$ and $d\tilde{\tau}_{1}=0$. In
the end we obtain an expression for the $3$-form $d^{2}v^{\flat }$ in terms
of $v$, $\tau _{1}$, $\tilde{\tau}_{1}$, $\tau _{7}$ and $\nabla \tau _{7}$.
Since the whole $3$-form must vanish, so must the components of the $3$-form
in ~$\Lambda _{1}^{3}$, $\Lambda _{7}^{3}$ and $\Lambda _{27}^{3}$. So let $%
\xi _{1}$ be the scalar corresponding to the $\Lambda _{1}^{3}$ component
and let $\xi _{7}$ and $\xi _{27}$ be the vector and the antisymmetric
symmetric tensor corresponding to the $\Lambda _{7}^{3}$ and $\Lambda
_{27}^{3}$ components of $d^{2}v^{\flat }$, respectively.

Then by considering the equations $\xi _{1}=0$, $\left( \xi _{7}\right)
^{a}v_{a}=0$ and $\left( \xi _{27}\right) _{mn}v^{m}v^{n}=0$, we can express 
$\left( \nabla _{a}\left( \tau _{7}\right) _{b}\right) v^{a}v^{b}$, $\nabla
^{a}\left( \tau _{7}\right) _{a}$ and $\left\vert \tau _{7}\right\vert ^{2}$
in terms of $M,$ $\tau _{1},\tilde{\tau}_{1}$ and $\left\langle \tau
_{7},v\right\rangle $. In particular, we find that 
\begin{eqnarray}
\ \ \ \ \ \ \ \ \ \left\vert \tau _{7}\right\vert ^{2} &=&\frac{%
3\left\langle \tau _{7},v\right\rangle ^{2}\left( 3M^{2}-10M+51\right) }{%
\left( 7M^{2}-66M-9\right) M}  \label{t17tt1t7sq1} \\
&&-\frac{4}{3}\frac{\left\langle \tau _{7},v\right\rangle \left( M+9\right)
^{2}\left( \tau _{1}\left( 1+M\right) ^{\frac{1}{3}}-4\tilde{\tau}%
_{1}\right) }{\left( 1+M\right) ^{\frac{1}{3}}\left( 7M^{2}-66M-9\right) } 
\notag \\
&&+\frac{2}{9}\frac{\tau _{1}^{2}M\left( M+9\right) ^{2}}{7M^{2}-66M-9}-%
\frac{16}{9}\frac{\tau _{1}\tilde{\tau}_{1}M\left( M+9\right) ^{2}}{\left(
1+M\right) ^{\frac{1}{3}}\left( 7M^{2}-66M-9\right) }  \notag \\
&&+\frac{32}{9}\frac{\tilde{\tau}_{1}^{2}M\left( M+9\right) ^{2}}{\left(
1+M\right) ^{\frac{2}{3}}\left( 7M^{2}-66M-9\right) }  \notag
\end{eqnarray}%
Further, we can consider the vector equations $\xi _{7}^{d}=0$, $\varphi
_{abc}v^{b}\xi _{7}^{c}=0,$ $\left( \xi _{27}\right) _{mn}v^{n}=0$ and $%
\varphi _{abc}$ $\left( \xi _{27}\right) _{\ \ n}^{b}v^{n}v^{c}=0.$ From
these, in particular, we find 
\begin{equation}
\tau _{7}=\frac{\left\langle \tau _{7},v\right\rangle }{M}v\ \ \text{and\ }%
\left\vert \tau _{7}\right\vert ^{2}=\frac{\left\langle \tau
_{7},v\right\rangle ^{2}}{M}.  \label{t17tt1t7sq2}
\end{equation}%
Equating (\ref{t17tt1t7sq1}) and (\ref{t17tt1t7sq2}), and solving for $%
\tilde{\tau}_{1}^{2}$, we obtain an expression for $\tilde{\tau}_{1}$ in
terms of $\tau _{1}$, $\tau _{7}$ and $v$.%
\begin{equation}
\tilde{\tau}_{1}=\frac{1}{4}\frac{\left( 1+M\right) ^{\frac{1}{3}}\left(
M\tau _{1}-3\left\langle \tau _{7},v\right\rangle \right) }{M}
\label{t17tt1tt1sol}
\end{equation}%
It can be checked that this expression for $\tilde{\tau}_{1}$ is in fact
consistent with the assumption $d\tilde{\tau}_{1}=0$.

Next, from equations $\left( \xi _{27}\right) _{ab}=0,$ $\varphi _{\ \ \
(a}^{cd}\left( \xi _{27}\right) _{b)d}v_{c}=0$ and $\varphi _{a}^{\ \
cd}\varphi _{b}^{\ \ \ ef}v_{c}v_{e}\left( \xi _{27}\right) _{df}$, we
finally obtain an expression for $\nabla _{a}\left( \tau _{7}\right) _{b}$.
Using (\ref{t17tt1t7sq2}) and (\ref{t17tt1tt1sol}) to eliminate $\left\vert
\tau _{7}\right\vert ^{2}$ and $\tilde{\tau}_{1}$ from the resulting
expression, we overall get: 
\begin{equation}
\nabla \tau _{7}=-\frac{\left\langle \tau _{7},v\right\rangle \left(
M-3\right) \left( M\tau _{1}+\left\langle \tau _{7},v\right\rangle \right) }{%
4M^{2}}g+\frac{\left( \tau _{1}^{2}M^{2}+6\left\langle \tau
_{7},v\right\rangle M\tau _{1}+3\left\langle \tau _{7},v\right\rangle
^{2}\right) }{6\left\langle \tau _{7},v\right\rangle ^{2}}\tau _{7}^{2}
\label{t17tt1dt7}
\end{equation}%
Now since $v$ is proportional to $\tau _{7}$, let us write $v=V\tau _{7}$
for some smooth function $V$. Then 
\begin{equation}
M=\left\vert v\right\vert ^{2}=V^{2}\left\vert \tau _{7}\right\vert ^{2}\ \ 
\text{and}\ \ \left\langle \tau _{7},v\right\rangle =V\left\vert \tau
_{7}\right\vert ^{2}  \label{t17tt1MV}
\end{equation}%
\qquad Thus we get the expressions (\ref{deltau7eq}) and (\ref{ttildecond})
for $\nabla \tau _{7}$ and $\tilde{\tau}_{1}$ in terms of $V$.
\end{proof}

Now it is easy to see that if $v$ is proportional to $\tau _{7}$ and $\tilde{%
\tau}_{1}$ satisfies (\ref{ttildecond}), then the equation (\ref%
{fulldelvt17tt1}) for $v$ is equivalent to the equation (\ref{deltau7eq})
for $\tau _{7}.$ However since these conditions are required for the
consistency of (\ref{fulldelvt17tt1}), the equation (\ref{fulldelvt17tt1})
is in fact equivalent to (\ref{deltau7eq}) together with conditions (\ref%
{vproptau7}) and (\ref{ttildecond}). This is now something that we can
solve, however for that we will need the following lemma.

\begin{lemma}[\protect\cite{CheegerColding}]
\label{LemWarpProd}Let $M$ be a $n$-dimensional Riemannian manifold. Then
the metric $g$ satisfies 
\begin{equation}
\nabla _{a}\nabla _{b}h=\lambda g_{ab}  \label{metconfhess}
\end{equation}%
for functions $h$ and $\lambda $ if and only if the underlying smooth
manifold is $\left( a,b\right) \times N$, for a $\left( n-1\right) $%
-dimensional manifold $N$, with a warped product metric $g$ given by%
\begin{equation}
g=\frac{dh^{2}}{\left\vert \nabla h\right\vert ^{2}}+\left\vert \nabla
h\right\vert ^{2}\hat{g}
\end{equation}%
where $\hat{g}$ is the induced metric on the $\left( n-1\right) $%
-dimensional slices.
\end{lemma}

\begin{theorem}
\label{ThmWarped}Consider a deformation of $\left( \varphi ,g\right) $ with $%
T_{ab}$ lying in the strict class $W_{1}\oplus W_{7}$ to $\left( \tilde{%
\varphi},\tilde{g}\right) $ with $\tilde{T}_{ab}$ lying in the class $\tilde{%
W}_{1}$. Then, such a deformation exists if and only if $M$ is a warped
product manifold $I\times _{f}N$ for some interval $I$ and $6$-dimensional
manifold $N\,$. There are three cases:

\begin{enumerate}
\item If $v^{\flat }=\frac{3}{\tau _{1}}\tau _{7}$, then for a $6$%
-dimensional metric $\hat{g}$, the original metric $g$ and new torsion $%
\tilde{\tau}_{1}$ must be given by 
\begin{eqnarray}
g &=&\frac{\tau _{7}^{2}}{\left\vert \tau _{7}\right\vert ^{2}}+\tau
_{1}^{-10}\left\vert \tau _{7}\right\vert ^{2}\hat{g}  \label{gwarped1a} \\
\tilde{\tau}_{1} &=&0
\end{eqnarray}

\item If $v^{\flat }=-\frac{3}{\tau _{1}}\tau _{7}$, then%
\begin{eqnarray}
g &=&\frac{\tau _{7}^{2}}{\left\vert \tau _{7}\right\vert ^{2}}+\tau
_{1}^{2}\left\vert \tau _{7}\right\vert ^{2}\hat{g}  \label{gwarped1} \\
\tilde{\tau}_{1}^{3} &=&\frac{\tau _{1}}{8}\left( \tau _{1}^{2}+9\left\vert
\tau _{7}\right\vert ^{2}\right) .  \label{ttildewarped1}
\end{eqnarray}

\item If $v^{\flat }=\frac{f}{\tau _{1}}\tau _{7}$ where $f^{2}=\frac{9A\tau
_{1}^{3}}{A\tau _{1}^{3}-1}$ for an arbitrary constant $A$, then 
\begin{eqnarray}
g &=&\frac{\tau _{7}^{2}}{\left\vert \tau _{7}\right\vert ^{2}}+\frac{\left(
f-3\right) ^{\frac{10}{3}}}{f^{\frac{2}{3}}\left( f+3\right) ^{\frac{2}{3}}}%
\left\vert \tau _{7}\right\vert ^{2}\hat{g}  \label{gwarped2} \\
\tilde{\tau}_{1}^{3} &=&\frac{1}{64}\frac{\left( f-3\right) ^{2}}{Af\left(
f^{2}-9\right) }\left( 1+\frac{9A\tau _{1}\left\vert \tau _{7}\right\vert
^{2}}{A\tau _{1}^{3}-1}\right)  \label{ttildewarped2}
\end{eqnarray}
\end{enumerate}
\end{theorem}

\begin{proof}
From the expression for $\nabla v$ (\ref{fulldelvt17tt1}), we find that 
\begin{equation}
dM=\frac{3}{2}\left( V^{3}\left\vert \tau _{7}\right\vert ^{2}\tau
_{1}+V^{2}\left\vert \tau _{7}\right\vert ^{2}+V\tau _{1}+1\right) \tau _{7}
\label{t17tt1dM1}
\end{equation}%
However, from (\ref{t17tt1MV}), 
\begin{equation}
dM=2V\left\vert \tau _{7}\right\vert ^{2}dV+2V^{2}\tau _{7}\lrcorner \left(
\nabla \tau _{7}\right)   \label{t17tt1dM2}
\end{equation}%
So equating (\ref{t17tt1dM1}) and (\ref{t17tt1dM2}), and using (\ref%
{deltau7eq}), we get an expression for $dV$:%
\begin{equation}
dV=\frac{1}{6}V\left( 3-V^{2}\tau _{1}^{2}\right) \tau _{7}  \label{t17tt1dV}
\end{equation}%
Now consider $d\left( V\tau _{1}\right) $. Using (\ref{t17tt1dV}) and the
fact that that $d\tau _{1}=\tau _{1}\tau _{7}$ we find that 
\begin{equation}
d\left( V\tau _{1}\right) =\frac{1}{6}\left( 9V\tau _{1}-V^{3}\tau
_{1}^{3}\right) \tau _{7}  \label{t17tt1dVt1}
\end{equation}%
Let $f=V\tau _{1}$, so we have the equation%
\begin{equation}
df=\frac{1}{6}\left( 9f-f^{3}\right) d\left( \log \tau _{1}\right) .
\label{dfdtau7}
\end{equation}%
Consider the constant solutions first. Since $\tau _{1}$ is non-zero, the
solution $f=0$ implies that $V=0$, and hence $v=0$, so this is a degenerate
solution. The non-trivial constant solutions are $f\equiv \pm 3.$ If $%
f=V\tau _{1}\equiv 3$, then in (\ref{ttildecond}), $\tilde{\tau}_{1}=0$. In
this case, from (\ref{deltau7eq}), we find 
\begin{equation}
\nabla \tau _{7}=\left( \frac{\tau _{1}^{2}}{3}-\left\vert \tau
_{7}\right\vert ^{2}\right) g+5\tau _{7}\otimes \tau _{7}  \label{delt7f3}
\end{equation}%
Now using (\ref{delt7f3}) and $d\tau _{1}=\tau _{1}\tau _{7}$ we can relate
the metric to the Hessian of a function via the following expression 
\begin{equation}
\nabla _{a}\nabla _{b}\left( \tau _{1}^{-5}\right) =-5\tau _{1}^{-5}\left( 
\frac{\tau _{1}^{2}}{3}-\left\vert \tau _{7}\right\vert ^{2}\right) g_{ab}.
\end{equation}%
Hence by Lemma \ref{LemWarpProd} we get (\ref{gwarped1a}). Now consider the
solution $f=V\tau _{1}\equiv -3$. In this case, from (\ref{ttildecond}) we
find 
\begin{equation}
V^{2}\left\vert \tau _{7}\right\vert ^{2}=-\frac{8\tilde{\tau}_{1}^{3}V^{3}}{%
27}-1.  \label{t17tt1vtay7}
\end{equation}%
Using the fact that $V=-\frac{3}{\tau _{1}}$ and (\ref{t17tt1vtay7}) in (\ref%
{deltau7eq}) we have 
\begin{equation}
\nabla \tau _{7}=\frac{2}{9\tau _{1}}\left( 2\tilde{\tau}_{1}^{3}-\tau
_{1}^{3}\right) g-\tau _{7}\otimes \tau _{7}  \label{t17tt1dt7b}
\end{equation}%
However, in the $W_{1}\oplus W_{7}$ class, $\tau _{7}=d\log \tau _{1}$, so
we rewrite (\ref{t17tt1dt7b}) in terms of $\tau _{1}$: 
\begin{equation}
\nabla _{a}\nabla _{b}\tau _{1}=\frac{2}{9}\left( 2\tilde{\tau}_{1}^{3}-\tau
_{1}^{3}\right) g_{ab}  \label{t17tt1hess1}
\end{equation}%
Hence, by Lemma \ref{LemWarpProd}, the metric must be (\ref{gwarped1}).
Using (\ref{t17tt1vtay7}), we also find that $\tilde{\tau}_{1}$ satisfies (%
\ref{ttildewarped1}).

Suppose now $f$ is non-constant. Then whenever the right-hand side of (\ref%
{dfdtau7}) is non-zero - that is, $f\neq \pm 3$, we can separate variables
in (\ref{dfdtau7}). Integrating, we obtain 
\begin{equation}
f^{2}=\frac{9A\tau _{1}^{3}}{A\tau _{1}^{3}-1}  \label{t17tt1xsol}
\end{equation}%
for some positive constant $A$. Suppose now $f=\pm 3$ at some point, then if 
$f$ is non-constant, we must have $f^{2}\longrightarrow 9$ in (\ref%
{t17tt1xsol}) for some values of $A$ and $\tau _{1}$. However, we can see
from (\ref{t17tt1xsol}) that this happens if and only if $\left\vert \tau
_{1}\right\vert \longrightarrow \infty $, but $\tau _{1}$ is smooth, so in
fact, either $f\equiv \pm 3$ (so these are singular solutions of \ref%
{dfdtau7}) or $f$ is nowhere equal to $\pm 3$ and is given by (\ref%
{t17tt1xsol}) everywhere. We have already covered the constant cases above,
so now we can assume that (\ref{t17tt1xsol}) holds everywhere. In particular
from (\ref{t17tt1xsol}) we also get the relations 
\begin{equation}
V^{2}=\frac{9A\tau _{1}}{A\tau _{1}^{3}-1}\ \text{and}\ \ V^{3}=Af\left(
f^{2}-9\right)  \label{vsqvcub}
\end{equation}%
From (\ref{ttildecond}), find can find $\left\vert \tau _{7}\right\vert ^{2}$
in terms of $f$ and $\tilde{\tau}_{1}^{3}$. Also, note that from (\ref%
{dfdtau7}) we get 
\begin{equation*}
\nabla _{a}\nabla _{b}f=-\frac{1}{6}f\left( f^{2}-9\right) \nabla _{a}\left(
\tau _{7}\right) _{b}+\frac{3\left( f^{2}-3\right) }{f\left( f^{2}-9\right) }%
\nabla _{a}f\nabla _{b}f.
\end{equation*}%
\qquad So overall, we can rewrite (\ref{deltau7eq}) as an equation for $f$
in the form 
\begin{equation}
\nabla _{a}\nabla _{b}f=P(f)g_{ab}+Q\left( f\right) \nabla _{a}f\nabla _{b}f
\label{exthessmet}
\end{equation}%
for some functions $P\left( f\right) $ and $Q\left( f\right) $. The exact
form of $P\left( f\right) $ is not very important, but $Q\left( f\right) $
is given by 
\begin{equation*}
Q\left( f\right) =\frac{2\left( f^{2}-3f-6\right) }{f\left( f^{2}-9\right) }.
\end{equation*}%
In order to reduce (\ref{exthessmet}) to the form (\ref{metconfhess}) we
need to find a function $F\left( f\right) $ that satisfies 
\begin{equation}
\frac{d^{2}F}{df^{2}}+\frac{dF}{df}Q\left( f\right) =0.  \label{d2Feq}
\end{equation}%
For such an $F$, the Hessian would be proportional to the metric. Let $G=%
\frac{dF}{df}$, then by separation of variables, we solve (\ref{d2Feq}) for $%
G$ and get the solution 
\begin{equation}
G=\frac{6\left( f-3\right) ^{\frac{2}{3}}}{f^{\frac{4}{3}}\left( f+3\right)
^{\frac{4}{3}}}
\end{equation}%
Note that we have chosen a constant factor for convenience. Since $f$ is
nowhere vanishing and is nowhere equal to $-3$, this can be integrated
further to find $F$. Hence for this $F$, $\nabla _{a}\nabla _{b}F$ is
proportional to $g_{ab}$. Therefore by Lemma \ref{LemWarpProd}, the metric
must be a warped product of the form 
\begin{equation}
g=\frac{1}{\left\vert \nabla F\right\vert ^{2}}dF^{2}+\left\vert \nabla
F\right\vert ^{2}\hat{g}  \label{warpg2}
\end{equation}%
Note that $\frac{dF}{\left\vert \nabla F\right\vert }=\frac{df}{\left\vert
\nabla f\right\vert }=\frac{\tau _{7}}{\left\vert \tau _{7}\right\vert }$,
and therefore, using (\ref{dfdtau7}), we get 
\begin{eqnarray*}
\left\vert \nabla F\right\vert ^{2} &=&\left( \frac{dF}{df}\right)
^{2}\left\vert \nabla f\right\vert ^{2}=\frac{1}{36}G^{2}\left\vert \tau
_{7}\right\vert ^{2}f^{2}\left( f^{2}-9\right) ^{2} \\
&=&\frac{\left( f-3\right) ^{\frac{10}{3}}}{f^{\frac{2}{3}}\left( f+3\right)
^{\frac{2}{3}}}\left\vert \tau _{7}\right\vert ^{2}
\end{eqnarray*}%
Thus we obtain the metric (\ref{gwarped2}). From (\ref{ttildecond}) we get
the expression (\ref{ttildewarped2}) for $\tilde{\tau}_{1}$ by substituting (%
\ref{vsqvcub}).
\end{proof}

It was shown in by Cleyton and Ivanov in \cite{CleytonIvanovConf} that by
considering a warped product of an open interval over a $6$-dimensional
nearly K\"{a}hler manifold, it is possible to obtain $7$-dimensional
manifolds with $G_{2}$-structures. Moreover, it has been shown that for such
a construction the $\tau _{14}$ torsion component ($\tau _{2}$ in the
notation of \cite{CleytonIvanovConf}) will always vanish, and moreover, it
is possible to find parameters such that the $\tau _{27}$ torsion component (%
$\tau _{3}$ in \cite{CleytonIvanovConf}) will also become zero, leaving only 
$\tau _{1}$ and $\tau _{7}$ non-zero ($\tau _{0}$ and $\tau _{1}$ in their
notation), which are given in terms of the warp factor. In particular, if
the metric is $dt^{2}+h\left( t\right) ^{2}\hat{g}$ for warp factor $h>0$,
then, using our conventions for the torsion components, 
\begin{equation}
\tau _{1}=h^{-1}\sigma \sin \theta \ \ \ \text{and }\tau _{7}=h^{-1}\left(
\sigma \cos \theta -h^{\prime }\right) dt  \label{warpedt1t7}
\end{equation}%
where $\sigma $ is a constant related to the scalar curvature of the $6$%
-dimensional manifold and $\theta \left( t\right) $ satisfies $\theta
^{\prime }=h^{-1}\sigma \sin \theta $. These are precisely the kind of
manifolds that appear as solutions in Theorem \ref{ThmWarped} - warped
product $7$-manifolds with torsion in $W_{1}\oplus W_{7}$. The warp factors
in Theorem \ref{ThmWarped} have to be consistent with the expressions (\ref%
{warpedt1t7}), and so for each case in Theorem \ref{ThmWarped} we get a
system of two first-order ODEs for $h$ and $\theta $. Given appropriate
initial conditions we can say that solutions exist, however the analysis of
these solutions is something to be investigated further. Therefore, we can
construct examples of $7$-manifolds with a conformally nearly parallel $%
G_{2} $-structure such that a non-infinitesimal deformation in $\Lambda
_{7}^{3}$ gives a $G_{2}$-structure in a strictly smaller torsion class.
Moreover, applying a conformal transformation (\ref{sconf}) we can obtain a
nearly parallel $G_{2}$-structure, for which a combination of a conformal
transformation and a $\Lambda _{7}^{3}$ deformation lead to another nearly
parallel $G_{2}$-structure. As we have seen, there are various non-trivial
relationships between different $G_{2}$-structures and it will be a subject
of further study whether it is possible find transformations between other $%
G_{2}$ torsion classes. In particular, so far we only have a grasp on
deformations that lie in $\Lambda _{1}^{3}$ and $\Lambda _{7}^{3}$, however
it is likely that deformations in $\Lambda _{27}^{3}$ could yield the most
interesting results.


\bibliographystyle{jhep-a}
\bibliography{refs2}

\end{document}